\documentclass[letterpaper,11pt]{amsart}

\usepackage{amsmath}
\usepackage{amssymb}
\usepackage{amsfonts}
\usepackage{amsthm}
\usepackage{amsopn}
\usepackage{enumerate}
\usepackage[all]{xy}
\usepackage{array}
\usepackage{mathrsfs}
\usepackage{tikz}
\usepackage{dsfont}
\usepackage{caption}
\usepackage{fullpage}
\usepackage{hyperref}
\hypersetup{
 colorlinks  = true,  
 urlcolor   = blue,  
 linkcolor  = blue,  
 citecolor  = blue   
}

\mathchardef\mhyphen="2D
\newcommand{\m}{\alpha}

\newcommand{\kk}{\mathds{k}}
\newcommand{\NN}{\ensuremath{\mathbb{N}}}
\newcommand{\NNp}{\ensuremath{\mathbb{N}^+}}
\newcommand{\ZZ}{\ensuremath{\mathbb{Z}}}

\newcommand{\FFF}{\ensuremath{\mathcal{F}}}
\newcommand{\GGG}{\ensuremath{\mathcal{G}}}
\newcommand{\SSS}{\ensuremath{\mathcal{S}}}
\newcommand{\III}{\ensuremath{\mathcal{I}}}
\newcommand{\PPPP}{\ensuremath{\mathcal{P}}}

\newcommand{\CCC}{\ensuremath{\mathsf{C}}}
\newcommand{\DDD}{\ensuremath{\mathsf{D}}}
\newcommand{\PPP}{\ensuremath{\mathsf{P}}}

\newcommand{\nn}{\ensuremath{\mathfrak{n}}}
\newcommand{\pp}{\ensuremath{\mathfrak{p}}}

\newcommand{\ra}{\ensuremath{\longrightarrow}}
\newcommand{\sra}{\ensuremath{\rightarrow}}

\newcommand{\Hom}{\operatorname{Hom}}
\newcommand{\uHom}{\operatorname{\underline{Hom}}}

\newcommand{\pd}{\operatorname{pd}}
\newcommand{\gldim}{\operatorname{gl dim}}

\newcommand{\ol}{\overline}
\newcommand{\st}{\ensuremath{\mathrel{}\middle|\mathrel{}}}

\newcommand{\Id}{\operatorname{Id}}
\newcommand{\im}{\operatorname{im}}
\newcommand{\coker}{\operatorname{coker}}
\newcommand{\dirlim}{\varinjlim}

\def\lcm{\mathop{\operatorname{lcm}}}

\newcommand{\gr}{\operatorname{gr}\mhyphen}

\newcommand{\qgr}{\operatorname{qgr}\mhyphen}
\newcommand{\grB}{\operatorname{gr}\mhyphen(B,\ZZ_\fin)}

\newcommand{\grR}{\operatorname{gr}\mhyphen(R,\Gamma)}
\newcommand{\grS}{\operatorname{gr}\mhyphen(S,G)}
\newcommand{\QgrA}{Q_{\mathrm{gr}}(A)}

\newcommand{\fin}{\mathrm{fin}}
\newcommand{\Aut}{\operatorname{Aut}}
\newcommand{\Pic}{\operatorname{Pic}}
\newcommand{\s}[1]{\ensuremath{\langle #1 \rangle}}
\DeclareMathOperator{\Spec}{Spec}
\DeclareMathOperator{\proj}{proj}

\newcommand{\tors}{\ensuremath{\operatorname{tors}}\mhyphen}
\newcommand{\fdim}{\ensuremath{\operatorname{fdim}}\mhyphen}
\newcommand{\coh}{\operatorname{coh}}

\numberwithin{equation}{section}

\newtheorem{theorem}[equation]{Theorem}
\newtheorem{corollary}[equation]{Corollary}
\newtheorem{lemma}[equation]{Lemma}
\newtheorem{proposition}[equation]{Proposition}

\theoremstyle{remark}
\newtheorem{remark}[equation]{Remark}

\theoremstyle{definition}
\newtheorem{definition}[equation]{Definition}

\newtheorem*{notation}{Notation}

\newtheorem*{acknowledgments}{Acknowledgments}

\begin{document}
\title{The noncommutative schemes of generalized Weyl algebras}
\author{Robert Won}
\address{Department of Mathematics, UCSD, La Jolla, CA 92093-0112, USA.}
\address{Permanent address: Department of Mathematics and Statistics, Wake Forest University, Winston-Salem, NC 27109, USA.}
\email{wonrj@wfu.edu}
\thanks{The author was partially supported by NSF grant DMS-1201572.}
\date{}

\maketitle

\begin{abstract}
The first Weyl algebra over $\kk$, $A_1 = \kk \langle x, y\rangle/(xy-yx - 1)$ admits a natural $\ZZ$-grading by letting $\deg x = 1$ and $\deg y = -1$. Smith showed that $\gr A_1$ is equivalent to the category of quasicoherent sheaves on a certain quotient stack. Using autoequivalences of $\gr A_1$, Smith constructed a commutative ring $C$, graded by finite subsets of the integers. He then showed $\gr A_1 \equiv \gr (C, \ZZ_\fin)$. In this paper, we prove analogues of Smith's results by using autoequivalences of a graded module category to construct rings with equivalent graded module categories. For certain generalized Weyl algebras, we use autoequivalences defined in a companion paper so that these constructions yield commutative rings.
\end{abstract}

\section{Introduction}

Throughout this paper, let $\kk$ be an algebraically closed field of characteristic zero. All rings will be $\kk$-algebras and all categories and equivalences will be $\kk$-linear.

In noncommutative projective geometry, the techniques of (commutative) projective algebraic geometry are generalized to the setting of noncommutative rings. In \cite{AZ}, Artin and Zhang defined the noncommutative projective scheme associated to a right noetherian $\NN$-graded $\kk$-algebra $A$ in terms of its graded module category. Let $\gr A$ be the category of finitely generated $\ZZ$-graded right $A$-modules. A graded right $A$-module $M$ is called \emph{torsion} if there exists $r \in \NN$ such that $M \cdot A_{\geq r} = 0$. Let $\tors A$ be the full subcategory of $\gr A$ generated by torsion modules. Define the quotient category $\proj A = \gr A /\tors A$ and let $\pi$ denote the canonical functor $\gr A \rightarrow \proj A$. Let $\SSS$ denote the shift functor on $\proj A$, induced by the shift functor on $\gr A$.

\begin{definition}[Artin-Zhang, \cite{AZ}] Let $A$ be a right noetherian graded $\kk$-algebra. The \emph{noncommutative projective scheme of $A$} is the triple
$(\proj A, \pi A, \SSS)$.
\end{definition}

Much of the literature on noncommutative projective schemes has focused on \emph{connected graded $\kk$-algebras}---$\NN$-graded $\kk$-algebras $A = \bigoplus_{i \in \NN} A_i$ with $A_0 = \kk$. Since the commutative polynomial ring under its usual grading satisfies these hypotheses, connected graded $\kk$-algebras are analogues of (quotients of) commutative polynomial rings. For a connected graded $\kk$-algebra $A$ of Gelfand-Kirillov (GK) dimension 2, we call $\proj A$ a \emph{noncommutative projective curve}. A \emph{noncommutative projective surface} is $\proj A$ for a connected graded $\kk$-algebra $A$ of GK dimension 3. 

In \cite{AS}, Artin and Stafford classified noncommutative projective curves. The classification of noncommutative projective surfaces is an active area of current research. The noncommutative analogues of $\kk[x,y,z]$ (so-called \emph{Artin-Schelter regular algebras} of dimension 3) are well-understood (see \cite{AS87, ATV91, steph, steph97}) and many other examples of noncommutative projective surfaces have been studied (see \cite{vdb, vdb2, rogalski, sierra2}).

The first Weyl algebra over $\kk$ is the ring $A_1 = \kk \langle x, y \rangle/(xy - yx - 1)$. While $A_1$ is not a connected graded ring, it admits a natural $\ZZ$-grading by letting $\deg x =1$ and $\deg y = -1$. In \cite{sierra}, Sierra studied $\gr A_1$, the category of finitely generated $\ZZ$-graded right $A_1$-modules. We can picture the graded simple right modules of $\gr A_1$ as follows. For every $\lambda \in \kk \setminus \ZZ$, there is a simple module $M_\lambda$. For each integer $n$, there are two simple modules $X\s{n}$ and $Y\s{n}$. We can therefore represent the graded simples by the affine line $\Spec \kk[z]$ with integer points doubled.

\begin{figure}[h!]
\centerline{
	\begin{tikzpicture}[ scale=1.6]	
			\node at (-1.7,.1)[label = above: {}]{};
		\draw[thick,<-](-3.9,0)--(-3.1,0);
		\fill[black](-3,-.1) circle (1pt);
		\fill[black](-3,.1) circle (1pt);
		\node at (-3,-.1)[label = below: $-3$]{};
		\draw[thick](-2.9,0)--(-2.1,0);
		\fill[black](-2,-.1) circle (1pt);
		\fill[black](-2,.1) circle (1pt);
		\node at (-2,-.1)[label = below: $-2$]{};
		\draw[thick](-1.9,0)--(-1.1,0);
		\fill[black](-1,-.1) circle (1pt);
		\fill[black](-1,.1) circle (1pt);
		\node at (-1,-.1)[label = below:$-1$]{};
		\draw[thick](-.9,0)--(-.1,0);
		\fill[black](0,-.1) circle (1pt);
		\fill[black](0,.1) circle (1pt);
		\node at (0,-.1)[label = below: $0$]{};
		\draw[thick](.1,0)--(.9,0);
		\fill[black](1,-.1) circle (1pt);
		\fill[black](1,.1) circle (1pt);
		\node at (1,-.1)[label = below: $1$]{};
		\draw[thick](1.1,0)--(1.9,0);		
		\fill[black](2,-.1) circle (1pt);
		\fill[black](2,.1) circle (1pt);
		\node at (2,-.1)[label = below: $2$]{};
		\draw[thick,](2.1,0)--(2.9,0);		
		\fill[black](3,-.1) circle (1pt);
		\fill[black](3,.1) circle (1pt);
		\node at (3,-.1)[label = below: $3$]{};
		\draw[thick,->](3.1,0)--(3.9,0);		
	\end{tikzpicture}
	}
\caption{The graded simple modules of $\gr A_1$.} \label{fig.grA}
\end{figure}
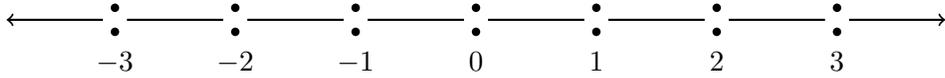

Sierra computed the Picard group---the group of autoequivalences modulo natural isomorphism---of this category. In particular, she constructed autoequivalences $\iota_n$ which permute the simple modules $X\s{n}$ and $Y\s{n}$ while fixing all other simples. Since the $\iota_n$ have order two, we call them \emph{involutions} and they generate a subgroup of the Picard group isomorphic to $(\ZZ/2\ZZ)^{(\ZZ)}$. We can identify this subgroup with $\ZZ_\fin$, the group of finite subsets of the integers with operation exclusive or. 

In \cite{smith}, Smith used Sierra's autoequivalences to construct a $\ZZ_\fin$-graded commutative ring $R$. For $J \in \ZZ_\fin$, define $\iota_J = \prod_{j \in J} \iota_j$. Let
\[ R = \bigoplus_{J \in \ZZ_\fin} \Hom_{\gr A_1}(A_1, \iota_J A_1).
\] 

\begin{theorem}[Smith, {\cite[Theorem 5.14]{smith}}] \label{smiththm} There is an equivalence of categories 
\[\gr A_1 \equiv \gr(R, \ZZ_\fin).\]
\end{theorem}

These results of Sierra and Smith suggest that there is interesting geometry within $\ZZ$-graded rings. In this paper, we prove analogues of Smith's results for generalized Weyl algebras by using autoequivalences of a category of graded modules to construct rings. Let $S$ be a $G$-graded ring and let $\Gamma$ be a subgroup of $\Pic(\gr S)$. In Proposition~\ref{autoring}, we give conditions on $\Gamma$ such that Smith's construction yields an associative ring $R$.

In a companion paper \cite{wonpic}, the author proved analogues of the results of \cite{sierra} to certain generalized Weyl algebras $A(f)$, with base ring $\kk[z]$, automorphism $\sigma: \kk[z] \to \kk[z]$ given by $\sigma(z) = z+1$ where $f = z(z+\m ) \in \kk[z]$ is a quadratic polynomial. These rings are primitive factors of the universal enveloping algebra of $\mathfrak{sl}(2,\kk)$.

In \cite{wonpic}, the author computed the Picard group of $\gr A(f)$, constructing analogues of Sierra's involutions $\iota_n$. In this paper, we prove analogues of Smith's results by using these involutions to define commutative rings. When $f$ has a multiple root, we construct the $\ZZ_\fin$-graded ring 
\[ B \cong \frac{\kk[z][b_n \mid n\in \ZZ]}{\left( b_n^2 = (z+n)^2 \mid n \in \ZZ \right)}.
\]
In Theorem~\ref{Bequivalent}, we prove that the category of graded modules over $B$ is equivalent to $\gr A(f)$.

When $f = z(z+\m )$ for some $\m \in \kk \setminus \ZZ$, we construct the $\ZZ_\fin \times \ZZ_\fin$-graded ring
\[ D \cong \frac{\kk[c_n, d_n \mid n\in \ZZ]}{\left( c_n^2 - n = c_m^2 - m, c_n^2 = d_n^2 - \m \mid m,n \in \ZZ\right)}
\]
and prove in Theorem~\ref{dDequivalent} that there is an equivalence of categories $\gr A(f) \equiv \gr (D, \ZZ_\fin \times \ZZ_\fin).$

Finally, when $f = z(z+\m )$ for some $\m \in \NN^+$, the existence of finite-dimensional $A(f)$-modules (in contrast with the previous two cases, in which all modules are infinite-dimensional) is an obstruction to a straightforward analogue of Smith's result. In this case, we consider the quotient category $\qgr A = \gr A/ \fdim A$ where $\fdim A$ is the full subcategory of $\gr A$ generated by finite-dimensional modules. When $R$ is a connected $\NN$-graded $\kk$-algebra, $\fdim R$ is the same as Artin and Zhang's $\tors R$. However, as generalized Weyl algebras do not satisfy these hypotheses, in general $\tors A$ is not the same as $\fdim A$. As an example, if $M$ is a finitely generated right-bounded module which is not left-bounded, then $M$ is in $\tors A$ but not $\fdim A$. In Corollary~\ref{cBequiv}, we prove that $\qgr A$ in this case is equivalent to $\gr A(z^2)$.

Although the ring $A\left( z (z+\m) \right)$ has finite global dimension, $\qgr A\left (z (z+\m \right)$ has infinite homological dimension, as $A \left(z^2\right)$ has infinite global dimension \cite[Theorem 5]{bav}. This is in contrast to the case of a noetherian connected graded domain $A$---when $A$ has finite global dimension, $\proj A$ has finite homological dimension. A consequence is that $\qgr A \left(z (z +\m) \right)$ has enough projectives (see Proposition~\ref{qgrgenerator}), which is another contrast to the connected graded case and with the quasicoherent sheaves on a projective variety.

This paper provides additional examples of the noncommutative geometry of $\ZZ$-graded rings. We discuss some possible directions for future work. The results in this paper are especially interesting in light of Artin and Stafford's classification of noncommutative projective curves  in \cite{AS}.
\begin{theorem}[Artin-Stafford, {\cite{AS}}] \label{ASthm}Let $A$ be a connected $\NN$-graded domain, generated in degree $1$ with GK dimension 2. Then there exists a projective curve $X$ such that $\qgr A \equiv \coh(X)$. 
\end{theorem}

Smith showed that $\gr A_1$ could be realized as the coherent sheaves on a quotient stack. Is there an analogue of Theorem~\ref{ASthm} for $\ZZ$-graded rings involving coherent sheaves on stacks? 

To close, we briefly summarize the organization of this paper. In section~\ref{prelims}, we establish notation and recall results on graded module categories over generalized Weyl algebras from \cite{wonpic}. In section~\ref{autoequivrings}, we use autoequivalences of a graded module category to construct rings with equivalent graded module categories. Finally, in section~\ref{gwaequiv}, we consider generalized Weyl algebras $A(f)$ and construct commutative rings whose graded module categories are equivalent to $\qgr A(f)$.

\begin{acknowledgments}The author was partially supported by NSF grant DMS-1201572. The research in this paper was completed while the author was a PhD student at the University of California, San Diego and, in some cases, extra details may be found in the author's PhD thesis \cite{won}. The author would like to thank Daniel Rogalski for his guidance, and Paul Smith, Sue Sierra, Cal Spicer, Perry Strahl, and Ryan Cooper for helpful conversations. The author would also like to thank the referee for suggesting many improvements.
\end{acknowledgments}

\section{Preliminaries} \label{prelims}
In this section, we briefly summarize the relevant results on graded module categories over generalized Weyl algebras from the companion paper \cite{wonpic}. 

\subsection{Notation}
We begin by fixing basic definitions, terminology, and notation. We follow the convention that 0 is a natural number so $\NN = \ZZ_{\geq 0}$. We use $\NN^+$ to denote the positive natural numbers. The notation $\ZZ_\fin$ denotes the group of finite subsets of the integers with operation exclusive or. We omit the braces on the singleton set $\{n\} \in \ZZ_\fin$ and refer to it simply as $n$.

If $\Gamma$ is an abelian semigroup, then we say a $\kk$-algebra $R$ is \emph{$\Gamma$-graded} if $R$ has a $\kk$-vector space decomposition
\[ R = \bigoplus_{\gamma \in \Gamma} R_\gamma
\]
such that $R_\gamma \cdot R_\delta \subseteq R_{\gamma +\delta}$ for all $\gamma, \delta \in \Gamma$. For a $\ZZ$-graded Ore domain $R$, we write $Q_{\mathrm{gr}}(R)$ for the graded quotient ring of $R$, the localization of $R$ at all nonzero homogeneous elements.

A \emph{graded right $R$-module} is an $R$-module $M$ with a $\kk$-vector space decomposition $M = \bigoplus_{\gamma \in \Gamma} M_\gamma$ such that $M_\gamma \cdot R_\delta \subseteq M_{\gamma + \delta}$ for all $\gamma, \delta \in \Gamma$. If $M$ and $N$ are graded right $R$-modules and $\delta \in \Gamma$, then a \emph{$\Gamma$-graded right $R$-module homomorphism of degree $\delta$} is an $R$-module homomorphism $\varphi$ such that $\varphi(M_\gamma) \subseteq N_{\gamma + \delta}$ for all $\gamma \in \Gamma$. If $\delta = 0$, then $\varphi$ is referred to as a \emph{$\Gamma$-graded right $R$-module homomorphism}. 

Define $\uHom_R(M,N)_\delta$ to be the set of all graded homomorphisms of degree $\delta$ from $M$ to $N$ and define
\[ \uHom_R(M,N) = \bigoplus_{\delta \in \Gamma} \uHom_R(M,N)_\delta.
\] 

The category of all finitely generated $\Gamma$-graded right $R$-modules with $\Gamma$-graded right $R$-module homomorphisms is an abelian category denoted $\gr (R, \Gamma)$. When $\Gamma = \ZZ$ we use the notation $\gr R$ in place of $\gr (R, \ZZ)$. When we are working in the category $\gr (R, \Gamma)$, we denote the $\Gamma$-graded right $R$-module homomorphisms by 
\[ \Hom_{\gr (R, \Gamma)}(M,N) = \uHom_R(M,N)_0.
\]

For a $\Gamma$-graded $R$-module $M$ and $\gamma \in \Gamma$, we denote the \emph{shifted module} $M\s{\gamma} = \bigoplus_{\delta \in \Gamma} M \s{\gamma}_\delta$ where $M\s{\gamma}_\delta = M_{\delta - \gamma}$. This is the opposite of the standard convention, but in keeping with the convention used in \cite{sierra} and \cite{wonpic}.

\subsection{Generalized Weyl algebras}
Fix $f \in \kk[z]$, let $\sigma: \kk[z] \sra \kk[z]$ be the automorphism given by $\sigma(z) = z+1$, and let 
\[A(f) = \frac{\kk[z] \langle x, y \rangle}{\begin{pmatrix}xz = \sigma(z)x & yz = \sigma^{-1}(z)y \\ xy = f & yx = \sigma^{-1}(f)\end{pmatrix}}.
\]
Then $A(f) = \kk[z]\left(\sigma, f\right)$ is a \emph{generalized Weyl algebra (GWA)} of degree 1 with base ring $\kk[z]$, defining element $f$, and defining automorphism $\sigma$. 

By \cite[Theorem 3.28]{bavjordan}, $A(f) \cong A(g)$ if and only if $f(z) = \eta g( \tau \pm z)$ for some $\eta, \tau \in \kk$ with $\eta \neq 0$. Hence, without loss of generality, we may assume that $f$ is monic and that $0$ is a root of $f$. We may also assume that $0$ is the largest integer root of $f$. The properties of $A(f)$ are determined by the distance between the roots of $f$. We say that two distinct roots, $\lambda$ and $\mu$ are \emph{congruent} if $\lambda - \mu \in \ZZ$. 

Generalized Weyl algebras were named by Bavula in \cite{bav}, in which he studied rings of the form $A(f)$ for $f$ of arbitrary degree. However, the study of GWAs of this form goes back at least as far as a paper of Joseph \cite{joseph}. Joseph determined that $A(f)$ is simple if and only if no two roots of $f$ are congruent. Hodges also studied these rings as noncommutative deformations of Type-A Kleinian singularities \cite{hodges}.

Bavula and Hodges independently proved that $A(f)$ is a noncommutative noetherian domain of Krull dimension $1$. They also determined that the global dimension of $A(f)$ depend on whether $f$ has multiple or congruent roots, as follows.
\begin{theorem}[Bavula {\cite[Theorem 5]{bav}} and Hodges {\cite[Theorem 4.4]{hodges}}] \label{gldimA} The global dimension of $A$ is equal to
\[ \gldim A = \begin{cases}\infty, & \text{if $f$ has at least one multiple root} \\ 2, & \text{if $f$ has no multiple roots but has congruent roots}; \\ 1, & \text{if $f$ has neither multiple nor congruent roots.} \end{cases}
\]
\end{theorem}

As in \cite{wonpic}, in this paper, we study GWAs $A(f)$ for quadratic polynomials $f$. As noted in \cite{hodges}, these algebras are infinite dimensional primitive factors of the enveloping algebra of $\mathfrak{sl}(2,\kk)$. Without loss of generality, $f = z(z+\m)$. When $\m = 0$, since $f$ has a multiple root, we say we are in the \emph{multiple root case}. When $\m \in \NNp $, we say that $f$ has congruent roots and we are in the \emph{congruent root case}. Finally, when $\m \in \kk \setminus \ZZ$, we say that $f$ has distinct non-congruent roots and refer to this case as the \emph{non-congruent root case}. When the case is clear from context, we call $A(f)$ simply $A$.

\subsection{The graded module category $\gr A$}
\label{catgrA}
Let $f = z(z+\m )$. It is straightforward to classify the simple graded right $A(f)$-modules (see \cite[Lemma 3.1]{wonpic}).
\begin{enumerate}
\item \label{case.double} If $\m = 0$, then up to graded isomorphism the graded simple $A(f)$-modules are:
\begin{itemize}
\item $X^f = A/(x,z)A$ and its shifts $X^f \s{n}$ for each $n \in \ZZ$;
\item $Y^f = \left(A/(y,z-1)A\right) \s{1}$ and its shifts $Y^f \s{n}$ for each $n \in \ZZ$;
\item $M^f_\lambda = A/(z+\lambda)A$ for each $\lambda \in \kk \setminus \ZZ $.
\end{itemize}

\item If $\m \in \NNp$, then up to graded isomorphism the graded simple $A(f)$-modules are:
\begin{itemize}
\item $X^f = \left(A/(x,z+\m )A\right) \s{-\m }$ and its shifts $X^f \s{n}$ for each $n \in \ZZ$;
\item $Y^f = \left(A/(y,z-1)A\right) \s{1}$ and its shifts $Y^f \s{n}$ for each $n \in \ZZ$;
\item $Z^f = A/(y^\m , x, z)A$ and its shift $Z^f \s{n}$ for each $n \in \ZZ$;
\item $M^f_\lambda = A/(z+\lambda)A$ for each $\lambda \in \kk \setminus \ZZ $.
\end{itemize}

\item If $\m \in \kk \setminus \ZZ$, then up to graded isomorphism the graded simple $A(f)$-modules are:
\begin{itemize}
\item $X^f_0 = A/(x,z)A$ and its shifts $X^f_0\s{n}$ for each $n \in \ZZ$;
\item $Y^f_0 = \left(A/(y,z-1)A \right) \s{1}$ and its shifts $Y^f_0\s{n}$ for each $n \in \ZZ$;
\item $X^f_\m = A/(x,z+\m )A$ and its shifts $X^f_\m \s{n}$ for each $n \in \ZZ$;
\item $Y^f_\m = \left(A/(y,z+\m -1)A\right) \s{1}$ and its shifts $Y^f_\m \s{n}$ for each $n \in \ZZ$;
\item $M^f_\lambda = A/(z+\lambda)A$ for each $\lambda \in \kk \setminus (\ZZ \cup \ZZ +\m )$.
\end{itemize}
\end{enumerate}

\begin{remark}When $f$ is clear from context, we suppress the superscript on these simple modules and call them $X$, $X_0$, $X_\m$, $Y$, $Y_0$, $Y_\m$, $Z$, and $M_\lambda$.\end{remark}

\begin{remark}We remark that when $\m = 0$ (case \ref{case.double} above), the ``picture" of $\gr A(f)$ looks superficially like that of $\gr A_1$ in Figure~\ref{fig.grA}. Though the simple modules are in bijection, the categories are certainly not equivalent. By \cite[Lemma 3.14]{wonpic}, $X^f$ and $Y^f$ have self-extensions in this case, while by \cite[Lemma 4.3]{sierra}, the corresponding modules in $\gr A_1$ do not.
\end{remark}

In \cite[\S 3.4]{wonpic}, the author studied the subcategory of $\gr A$ consisting of the rank one projectives. Each rank one projective in $\gr A$ is isomorphic to a graded submodule of $\QgrA = \kk(z)[x^{\pm 1}; \sigma]$, the graded quotient ring of $A$. Since every module in $\gr A$ has a natural left $\kk[z]$-action, making it a $\kk[z]$-$A$-bimodule, if $I$ is any graded submodule of $\QgrA$, we can write 
\[ I = \bigoplus_{i\in \ZZ}(a_i)x^i
\]
where $(a_i)$ denotes the $\kk[z]$-submodule of $\kk(z)$ generated by $a_i$.
We define the \emph{structure constants} of $I$ to be the sequence $\{c_i \mid i \in \ZZ\}$ where $c_i = a_i a_{i+1}^{-1}$. Because $I$ is an $A$-module, for each $i \in \ZZ$, $c_i \in \kk[z]$ and $c_i \mid \sigma^{i}(f)$. The author showed that a graded submodule of $\QgrA$ is determined up to isomorphism by its structure constants, and the properties of $I$ can be deduced from its structure constants.

Further, for a \emph{projective} graded submodule of $\QgrA$, $P$, more can be said about what simple modules $P$ surjects onto.

\begin{lemma}[{\cite[Corollary 3.33]{wonpic}}]\label{projXYZ} Let $P$ be a rank one graded projective $A$-module. Let $n \in \ZZ$.
\begin{itemize}\item If $\m = 0$, then $P$ surjects onto exactly one of $X\s{n}$ and $Y\s{n}$.
\item If $\m \in \NNp$, then $P$ surjects onto exactly one of $X\s{n}$, $Y\s{n}$, and $Z\s{n}$.
\item If $\m \in \kk \setminus \ZZ$, then $P$ surjects onto exactly one of $X_0\s{n}$ and $Y_0\s{n}$. Likewise, P surjects onto exactly one of $X_{\m}\s{n}$ and $Y_{\m}\s{n}$.
\end{itemize}
\end{lemma}

\begin{notation} By the above result, we denote by $F_j(P)$ the unique element of $\{X\s{j}, Y\s{j}, Z\s{j}\}$ which is a factor of $P$.
\end{notation}

%Combining the two previous lemmas, we can explicitly determine structure constants of a projective module by its simple factors. In the multiple root case, if $F_n(P) = X\s{n}$, then $c_n = 1$ and if $F_n(P) = Y\s{n}$ then $c_n = \sigma^{n}(z^2)$. In the congruent root case, the structure constants were computed in \cite[Table 1]{wonpic}. Using this, the author was also able to construct projective modules with specified simple factors.
%\begin{lemma}[{\cite[Lemma 3.36]{wonpic}}] \label{anyproj}
%Let $\m \in \NN$. For each $n \in \ZZ$ choose $S_n \in \{X\s{n}, Y\s{n}, Z\s{n}\}$ such that for $n \gg 0$, $S_n = X\s{n}$ and $S_{-n} = Y\s{-n}$. Then there exists a rank one graded projective $P$ such that $F_n(P) = S_n$ for all $n$.
%\end{lemma}

Given a rank one projective module $P$ with structure constants $\{c_i\}$, we recall that there is a \emph{canonical representation of $P$} given by
\[ \bigoplus_{i \in \ZZ} (p_i) x^i \subseteq Q_{\mathrm{gr}}(A).
\]
where $p_n = \prod_{j \geq n} c_j$. There is an isomorphism $\bigoplus_{i \in \ZZ} (p_i) x^i \cong P$ and we call this module a \emph{canonical rank one projective module}.

Finally, to understand the full subcategory of rank one projective modules, we must also understand the morphisms between them. In \cite[Proposition 3.38]{wonpic}, the author showed that if $P$ and $Q$ are finitely generated graded rank one projective $A$-modules embedded in $\QgrA$, then there is some $\theta \in \kk(z)$ such that $\Hom_{\gr A}(P,Q) = \theta \kk[z]$, acting by left-multiplication. We call $\theta$ a \emph{maximal embedding} $P\sra Q$. For canonical rank one projectives, the author computed this maximal embedding explicitly in \cite[Lemma 3.40]{wonpic}.

%\begin{lemma}[{\cite[Lemma 3.40]{wonpic}}] \label{projmaxemb} Let $P$ and $Q$ be rank one graded projective $A$-modules with structure constants $\{c_i\}$ and $\{d_i\}$, respectively. Write 
%\[ P = \bigoplus_{i \in \ZZ} (p_i) x^i = \bigoplus_{i \in \ZZ} \left( \prod_{j \geq i} c_j \right) x^i \mbox{ and } Q = \bigoplus_{i \in \ZZ}(q_i)x^i = \bigoplus_{i \in \ZZ} \left( \prod_{j \geq i} d_j \right) x^i.
%\]
%Then the maximal embedding $P \sra Q$ is given by multiplication by 
%\[
%\theta_{P,Q} = \lcm_{i \in \ZZ} \left( \frac{q_i}{\gcd(p_i, q_i)}\right)= \lcm_{i \in \ZZ} \left( \frac{\prod_{j \geq i}d_j}{\gcd\left( \prod_{j \geq i}c_j, \prod_{j \geq i} d_j\right)} \right)
%\]
%where $\lcm$ is the unique monic least common multiple $\theta_{P,Q} \in \kk[z]$.
%\end{lemma}

\subsection{The Picard group of $\gr A$}
For a graded module category $\gr A$, we denote the group of autoequivalences of $\gr A$ by $\Aut (\gr A)$. The \emph{Picard group of $\gr A$}, $\Pic(\gr A)$ is the group of autoequivalences of $\gr A$ modulo natural isomorphism. In \cite{wonpic}, the author showed that for a GWA $A(f)$ defined by quadratic polynomial $f \in \kk[z]$, $\Pic(\gr A(f)) \cong \ZZ_\fin \rtimes D_{\infty}$. 

We recall that in the congruent or double root case, an autoquivalence $\FFF$ is called \emph{numerically trivial} if $\{ \FFF(X\s{n}), \FFF(Y\s{n}) \} = \{ X \s{n}, Y \s{n} \}$. Similarly, in the non-congruent root case, $\FFF$ is called \emph{numerically trivial} if $\{ \FFF(X_0\s{n}), \FFF(Y_0\s{n}) \} = \{ X_0 \s{n}, Y_0 \s{n} \}$ and $\{ \FFF(X_\m \s{n}), \FFF(Y_\m \s{n}) \} = \{ X_\m \s{n}, Y_\m \s{n} \}$. The subgroup of $\Pic(\gr A)$ isomorphic to $\ZZ_\fin$ is generated by numerically trivial autoequivalences called \emph{involutions}. 

First, let $\m \in \NN$. In \cite[Proposition 5.13]{wonpic}, numerically trivial autoequivalences $\iota_j$ of $\gr A$ are constructed for each $j \in \ZZ$ such that $\iota_j(X \s{j}) \cong Y \s{j}$, $\iota_j(Y\s{j}) \cong X \s{j}$, and $\iota_j(S) \cong S$ for all other simple modules $S$.We define the autoequivalence
\[ \iota_J = \prod_{j \in J} \iota_j,
\]
with $\iota_{\emptyset} = \Id_{\gr A}$. Each $\iota_J$ is a subfunctor of $\Id_{\gr A}$, so acts on morphisms by restriction. In particular, for a rank one projective $P$, $\iota_j P$ is given by the kernel of a morphism to a particular indecomposable module. 

Analogous involutions exist in the non-congruent root case, as well (see \cite[Proposition 5.14]{wonpic}). Let $\m  \in \kk \setminus \ZZ$. Then for any $j \in \ZZ$, there is a numerically trivial autoequivalence $\iota_{(j, \emptyset)}$ of $\gr A$ such that $\iota_{(j, \emptyset)}(X_0 \s{j}) \cong Y_0 \s{j}$, $\iota_{(j, \emptyset)}(Y_0 \s{j}) \cong X_0\s{j}$, and $\iota_{(j, \emptyset)}(S) \cong S$ for all other simple modules $S$. Similarly, for any $j \in \ZZ$, there is a numerically trivial autoequivalence $\iota_{(\emptyset, j)}$ of $\gr A$ such that $\iota_{(\emptyset, j)}(X_\m \s{j}) \cong Y_\m \s{j}$, $\iota_{(\emptyset, j)}(Y_\m \s{j}) \cong X_\m\s{j}$, and $\iota_{(\emptyset, j)}(S) \cong S$ for all other simple modules $S$.

For each $(J,J') \in \ZZ_\fin \times \ZZ_\fin$, we define the autoequivalence
\[ \iota_{(J,J')} = \prod_{j \in J} \iota_{(j,\emptyset)} \prod_{j \in J'} \iota_{(\emptyset, j)},
\]
with $\iota_{(\emptyset, \emptyset)} = \Id_{\gr A}$. Each $\iota_{(J,J')}$ is a subfunctor of $\Id_{\gr A}$. As before, for a rank one projective $P$, both $\iota_{(\emptyset, j)} P$ and $\iota_{(j, \emptyset)} P$ are given by the kernel of a homomorphism to certain indecomposable modules.

\section{Defining rings from autoequivalences of categories}
\label{autoequivrings}
\subsection{Functors via bigraded bimodules}

In \cite{delrio}, Angel del R\'{i}o studies equivalences between graded module categories over graded rings. We establish notation in order to use del R\'{i}o's result, closely following the treatment found in the discussion before \cite[Theorem 5.13]{smith}. Let $R$ and $S$ be $\kk$-algebras graded by the abelian groups $\Gamma$ and $G$, respectively. Following the definition of del R{\'i}o in \cite{delrio}, we define a \emph{bigraded $R$-$S$-bimodule} to be an $R$-$S$-bimodule $P$ with a $\kk$-vector space decomposition
\[ P = \bigoplus_{(\gamma, g) \in \Gamma \times G} P_{(\gamma, g)}
\] 
that respects the graded structure of $R$ on the left and $S$ on the right. That is, for any $\gamma, \delta \in \Gamma$ and any $g, h \in G$,
\[ R_{\gamma} \cdot P_{(\delta, h)} \cdot S_{g} \subseteq P_{(\gamma + \delta, g+h)}.
\]
When we want to specify the degrees of an element $p \in P_{(\gamma, g)}$, we use the notation ${}^\gamma p^g$.

For any $\gamma \in \Gamma$ we have the $G$-graded right $S$-module
\[ P_{(\gamma, *)} = \bigoplus_{g \in G} P_{(\gamma, g)}.
\]
Note that if $r \in R_{\delta}$, then multiplication by $r$ is an $S$-module homomorphism $P_{(\gamma, *)} \sra P_{(\gamma + \delta, *)}$ that preserves $G$-degree and hence we get a $\kk$-linear map
\begin{equation} \label{Rgammamap} \varphi: R_{\delta} \sra \Hom_{\grS}(P_{(\gamma,*)}, P_{(\gamma + \delta, *) } ).
\end{equation}

We now define a functor $H_S(P, -): \grS \sra \grR$. If $M$ is a $G$-graded right $S$-module, let
\[ H_S(P,M) = \bigoplus_{\gamma \in \Gamma} \Hom_{\grS}(P_{(-\gamma,*)}, M).
\]
First, note that $H_S(P,M)$ is $\Gamma$-graded. The right $R$-module structure is given as follows : given $h \in \Hom_{\grS}(P_{(-\gamma, *)}, M)$ and $r \in R_{\delta}$, recall that as in equation \eqref{Rgammamap}, multiplication by $r$ gives a $G$-graded $S$-module homomorphism $\varphi(r): P_{(-\gamma - \delta, *)} \sra P_{(-\gamma,*)}$. Then
\[ h\cdot r = h\circ \varphi(r) \in \Hom_{\grS}(P_{(-\gamma - \delta, *) } , M ).
\]
We obtain a functor
\begin{equation} \label{HSfunctor} H_S(P,-): \grS \sra \grR
\end{equation}
by defining $H_S(P,-)$ on a morphism $h:M\sra N$ to be composition with $h$. In his discussion before \cite[Proposition 2]{delrio}, del R{\'i}o notes that $H_S(P,-)$ is naturally isomorphic to the functor he denotes $(-)_*^P$. 

In the subsequent sections we will attempt to construct a graded commutative ring $B$ and a bigraded $B$-$A$-bimodule $P$, and then use del R{\'i}o's theorem to prove that $H_{A(f)}(P,-)$ is an equivalence of categories. In the cases of a multiple root or distinct non-congruent roots, we will be able to construct such a ring and bimodule. In the case of congruent roots, we will pass to the quotient category $\qgr A(f)$ obtained by taking $\gr A(f)$ modulo its full subcategory of finite-dimensional modules. We then show that $\qgr A(f) \equiv \grB$ for a commutative ring $B$. In the next section, we develop machinery which constructs $\Gamma$-graded rings $R$ from autoequivalences in the Picard group of $\gr (S, G)$. 

\subsection{Subgroups of autoequivalences}
\label{ringconstr}
Let $G$ be an abelian group and let $S$ be a $G$-graded $\kk$-algebra. Suppose we have a subgroup $\Gamma \subseteq \Pic(\gr (S,G))$ and for each $\gamma \in \Gamma$, choose one autoequivalence in the equivalence class of $\gamma$, $\FFF_{\gamma} \in \Aut(\gr (S,G))$. Since $\Pic(\gr (S,G))$ is the group $\Aut (\gr(S,G))$ modulo natural isomorphism, we have, for all $\gamma, \delta \in \Gamma$
\[\FFF_{\gamma} \FFF_{\delta} \cong \FFF_{\delta} \FFF_{\gamma} \cong \FFF_{\gamma + \delta}.\]
Choose a natural isomorphism $\eta_{\gamma, \delta}$ from $\FFF_\gamma \FFF_\delta$ to $\FFF_{\gamma + \delta}$ and let $\Theta_{\gamma, \delta}$ be $\eta_{\gamma, \delta}$ at $S$. Then $\Theta_{\gamma, \delta}$ is a $G$-graded $S$-module isomorphism $\FFF_{\gamma} \FFF_{\delta} S \sra \FFF_{\gamma + \delta} S$. Motivated by Smith's construction in \cite[\S 10]{smith}, we can define a $\Gamma$-graded ring $R$ if the isomorphisms $\Theta_{\gamma, \delta}$ satisfy the following condition: for all $\gamma, \delta, \epsilon \in \Gamma$ and for all $\varphi \in \Hom_{\gr (S,G)}(S, \FFF_\gamma S)$

\begin{equation} \label{thetacond} \Theta_{\epsilon, \gamma + \delta} \circ \FFF_{\epsilon}(\Theta_{\delta, \gamma}) \circ \FFF_{\epsilon} \FFF_{\delta} (\varphi) = \Theta_{\delta+ \epsilon, \gamma} \circ \FFF_{\delta+ \epsilon}(\varphi) \circ \Theta_{\epsilon, \delta}.
\end{equation}
Morally, this says that the map $\Aut (\gr(S,G)) \times \Aut (\gr(S,G)) \sra \Aut (\gr(S,G))$ mapping $(\FFF_{\gamma} , \FFF_{\delta})$ to $\FFF_{\gamma + \delta}$ is associative, although it is a weaker condition, since the isomorphisms $\Theta_{\gamma,\delta}$ are only at the module $S_S$.

\begin{proposition} \label{autoring} Assume the setup and notation above. If the autoequivalences $\{\FFF_{\gamma} \mid \gamma \in \Gamma\}$ and isomorphisms $\{\Theta_{\gamma,\delta} \mid \gamma,\delta \in \Gamma\}$ satisfy condition~\eqref{thetacond} then
\[ R = \bigoplus_{\gamma \in \Gamma} \Hom_{\gr (S,G)}(S, \FFF_{\gamma} S)
\]
is an associative $\Gamma$-graded ring with multiplication defined as follows. For $\varphi \in R_{\gamma}$ and $\psi \in R_{\delta}$ set
\[ \varphi \cdot \psi = \Theta_{\delta, \gamma} \circ \FFF_{\delta} (\varphi) \circ \psi.
\]
\end{proposition}
\begin{proof} We need to check that the multiplication defined above is associative. It suffices to check on homogeneous elements, so let $\varphi, \psi, \xi \in R$ be homogeneous elements of degree $\gamma$, $\delta$, and $\epsilon$, respectively. Then, by definition
\[ (\varphi \cdot \psi) \cdot \xi = \Theta_{\epsilon, \gamma + \delta} \circ \FFF_{\epsilon} (\Theta_{\delta,\gamma}) \circ \FFF_{\epsilon} \FFF_{\delta} (\varphi) \circ \FFF_{\epsilon}(\psi) \circ \xi
\]
and 
\[ \varphi \cdot (\psi \cdot \xi )= \Theta_{\delta+ \epsilon, \gamma} \circ \FFF_{\delta+ \epsilon}(\varphi) \circ \Theta_{\epsilon, \delta} \circ \FFF_{\epsilon}(\psi) \circ \xi.
\]
Since we assumed the isomorphisms satisfied condition~\eqref{thetacond}, $R$ is associative.
\end{proof}

If, for example, $S$ is a $\ZZ$-graded ring and we take $\Gamma = \ZZ$ generated by the shift functor $\SSS$ on $\gr S$, then $\SSS^n \SSS^m = \SSS^{n+m}$ so the isomorphisms $\Theta_{n,m}$ are trivial and condition~\eqref{thetacond} is automatic. By the construction in Proposition~\ref{autoring}, we recover our original ring as $R \cong S$. Since we found (in \cite[Propositions 5.13 and 5.14]{wonpic}) that for a GWA $A(f)$ defined by a quadratic polynomial $f$, the category $\gr A(f)$ has many autoequivalences, we can choose a more interesting subgroup of autoequivalences to define such a ring. 

Given rings $R$ and $S$ as in Proposition~\ref{autoring}, condition~\eqref{thetacond} also allows us to define a bigraded $R$-$S$-bimodule $P$. Let
\[ P = \bigoplus_{\gamma \in \Gamma} \FFF_{\gamma} S.
\]
The $G$-graded right $S$-module structure on each $\FFF_{\gamma} S$ gives $P$ a $G$-graded right $S$-module structure. Let $\varphi \in R_{\gamma}$ and $p \in P_{(\delta,*)} = \FFF_{\delta} S$. Then $P$ has a $\Gamma$-graded left $R$-module structure given by
\begin{equation}\label{RactionP} \varphi \cdot p = \left[ \Theta_{\delta,\gamma} \circ \FFF_{\delta}(\varphi)\right] (p) \in P_{(\gamma+\delta,*)}.
\end{equation}
For $\varphi \in R_{\gamma}$, $\psi \in R_{\delta}$,
\[ (\psi \cdot \varphi) \cdot p = \left[\Theta_{\epsilon, \gamma + \delta} \circ \FFF_{\epsilon}(\Theta_{\delta, \gamma}) \circ \FFF_{\epsilon} \FFF_{\delta} (\varphi) \circ \FFF_{\epsilon}(\psi)\right](p)
\]
while
\[
\psi \cdot(\varphi \cdot p) = \left[\Theta_{\delta+ \epsilon, \gamma} \circ \FFF_{\delta+ \epsilon}(\varphi) \circ \Theta_{\epsilon, \delta} \circ \FFF_{\epsilon}(\psi) \right](p).
\]
and since we assumed condition~\eqref{thetacond}, therefore equation~\eqref{RactionP} gives $P$ a $\Gamma$-graded left $R$-module structure. Since $\FFF_{\delta}$ is an autoequivalence of $\gr (S,G)$ and $\Theta_{\delta,\gamma}$ is a $G$-graded $S$-module isomorphism, it is easily checked that this makes $P$ a bigraded $R$-$S$-bimodule. This extra left $R$-module structure makes $H_S(P,P)$ a bigraded $R$-$R$-bimodule.

As a graded ring, $R$ has its usual $R$-$R$-bimodule structure. We define the bigraded $R$-$R$-bimodule 
\[\hat{R} = \bigoplus_{\gamma \in \Gamma} \bigoplus_{\delta \in \Gamma} \hat{R}_{(\gamma,\delta)} = \bigoplus_{\gamma \in \Gamma} \bigoplus_{\delta \in \Gamma} R_{\gamma + \delta} = \bigoplus_{\gamma \in \Gamma} \bigoplus_{\delta \in \Gamma} \Hom_{\gr (S,G)}(S, \FFF_{\gamma + \delta} S)
\] 
and note that there exists a canonical homomorphism of bigraded $R$-$R$-bimodules
\[ \varrho_R^P: \hat{R} \ra H_S(P,P) = \bigoplus_{\gamma \in \Gamma} \Hom_{\gr (S,G)}\left(\FFF_{-\gamma} S, \bigoplus_{\delta \in \Gamma} \FFF_{\delta} S \right) 
\]
where $\varrho_R^P$ maps the element ${}^{\gamma}\varphi^{\delta}$ to the homomorphism which maps $p \in \FFF_{-\gamma} S$ to $\varphi \cdot p \in \FFF_{\delta} S$, where $\varphi$ acts as in equation~\eqref{RactionP}, and maps all other homogeneous elements to $0$. This map $\varrho_R^P$ is the same as the one del R{\'i}o calls $\varrho_A^P$ in \cite[Lemma 5]{delrio}, though del R{\'i}o's bigraded bimodule's module structures are on opposite sides.

\begin{proposition}\label{RSequiv} Assume the setup and notation above. Suppose that the autoequivalences $\{\FFF_{\gamma} \mid \gamma \in \Gamma\}$ and isomorphisms $\{\Theta_{\gamma,\delta} \mid \gamma,\delta \in \Gamma\}$ satisfy condition~\eqref{thetacond} and let $R$ and $S$ be as in Proposition~\ref{autoring}. If $P = \bigoplus_{\gamma \in \Gamma} \FFF_{\gamma} S$ is a generator of $\gr (S, G)$, then $H_S(P,-)$ gives an equivalence of categories 
\[\gr (R, \Gamma) \sra \gr (S,G).\]
\end{proposition}
\begin{proof}This follows immediately from \cite[Theorem 7(c)]{delrio} as long as $P$ is a projective generator of $\gr (S,G)$ and $\varrho_R^P$ is an isomorphism. Since each $\FFF_{\gamma}$ is an autoequivalence, $\FFF_{\gamma}S$ is automatically projective so if $P$ is a generator then it is a projective generator. Hence, we need only show that $\varrho_R^P$ is an isomorphism. Assuming the setup above, recall that for an element
\[{}^\gamma \varphi^\delta \in {}_{\gamma} R_{\delta} = \Hom_{\gr (S,G)}(S, \FFF_{\gamma + \delta} S)\]
$\varrho_R^P(\varphi)$ is given by the homomorphism $\Theta_{\delta,\gamma} \circ \FFF_{\delta}(\varphi): \FFF_{-\gamma} S \sra \FFF_{\delta} S$. Now since $\FFF_{-\gamma}$ is an autoequivalence it gives an isomorphism
\[\Hom_{\gr (S,G)}( S, \FFF_{\gamma + \delta} S) \cong \Hom_{\gr (S,G)}( \FFF_{-\gamma}S,\FFF_{-\gamma} \FFF_{\gamma + \delta} S)
\]
and $\Theta_{-\gamma,\gamma+\delta}$ gives an isomorphism $\FFF_{-\gamma}\FFF_{\gamma+\delta}S \sra \FFF_{\delta} S$. Hence, $\varrho_R^P$ is an isomorphism and $H_S(P,-)$ is an equivalence of categories.
\end{proof}

With this framework in place, we need only find subgroups of $\Pic(\gr (S,G))$ such that the autoequivalences $\{\FFF_{\gamma} \mid \gamma \in \Gamma\}$ satisfy condition~\eqref{thetacond} and the $\FFF_{\gamma} S$ generate $\gr (S,G)$. This machinery then yields a $\Gamma$-graded ring $R$ such that $\gr(R,\Gamma) \equiv \gr(S,G)$. For a generalized Weyl algebra $A(f)$, we will see that the involutions constructed in \cite{wonpic} often satisfy these conditions. 
By \cite[Lemma 5.17]{wonpic}, we know when the $\iota_J A$ or $\iota_{(J,J')}A$ generate $\gr A$.

\section{Generalized Weyl algebras defined by quadratic polynomials}
\label{gwaequiv}

Throughout this section, we let $A(f)$ denote the generalized Weyl algebra with base ring $\kk[z]$ automorphism $\sigma(z) = z+1$ and quadratic polynomial $f = z(z+\m )\in \kk[z]$. We divide into three cases: first we consider the multiple root case when $\m  = 0$, then the non-congruent root case when $\m  \in \kk \setminus \ZZ$, and finally the congruent root case when $\m \in \NN^+$.

\subsection{Multiple root}
\label{mthcr}

Let $\m = 0$. The autoequivalences $\{\iota_n \mid n \in \ZZ\}$ formed a subgroup of $\Pic(\gr A)$ isomorphic to $\ZZ_\fin$. By \cite[Lemma 5.17]{wonpic}, we know that the set $\{\iota_J A \mid J \in \ZZ_{\fin}\}$ generates $\gr A$. We will show that these autoequivalences satisfy condition~\eqref{thetacond}, and hence we can construct a $\ZZ_\fin$-graded commutative ring $B$ such that $\gr(B, \ZZ_{\fin}) \equiv \gr A$. For each $J \in \ZZ_{\fin}$, we define the polynomial
\[h_J = \prod_{ j \in J } (z+j)^2 .
\]
For completeness, define $h_\emptyset = 1$. In \cite[Proposition 5.13]{wonpic} the author showed that for a projective module $P$, $\iota_n^2 P = (z+n)^2 P$ and so $\iota_n^2 \cong \Id_{\gr A}$. We denote by $\sigma_n$ the isomorphism $\iota_n^2 A \sra A$. Since $A$ is projective, $\sigma_n$ is given by left multiplication by $h_n^{-1}$. Similarly, for $J \in \ZZ_\fin$, we define
\[ \sigma_J: \iota_J^2 A \sra A
\]
given by left multiplication by $h_J^{-1}$. Now, for $I, J\in \ZZ_{\fin}$, we define
\[\Theta_{I,J} = \iota_{I \oplus J}(\sigma_{I \cap J}) = \Theta_{J,I}
\]
\[ \Theta_{I,J}: \iota_I \iota_J A = \iota_{I\oplus J} \iota_{I \cap J}^2 A \sra \iota_{I\oplus J} A.
\]

\begin{lemma} \label{Btheta} The isomorphisms $\{\Theta_{I,J} \mid I,J \in \ZZ_{\fin} \}$ and the autoequivalences $\{\iota_K \mid K \in \ZZ_{\fin}\}$ satisfy condition~\eqref{thetacond}.
\end{lemma}
\begin{proof}We must show that for all $I, J, K \in \ZZ_{\fin}$ and all $\varphi \in \Hom_{\gr A}(A, \iota_I A)$,
\[\Theta_{K, I \oplus J} \circ \iota_{K}(\Theta_{J, I}) \circ \iota_{K} \iota_{J} (\varphi) = \Theta_{J+ K, I} \circ \iota_{J+ K}(\varphi) \circ \Theta_{K, J}
\]
or equivalently that
\begin{align*} &\iota_{I \oplus J \oplus K}(\sigma_{(I \oplus J) \cap K}) \circ \iota_{K} \iota_{I \oplus J}(\sigma_{I \cap J}) \circ \iota_{K} \iota_{J}(\varphi) \\ = &\iota_{I, J \oplus K}(\sigma_{I \cap (J \oplus K)}) \circ \iota_{J \oplus K}(\varphi) \circ \iota_{J \oplus K}(\sigma_{J \cap K})
\end{align*}

Recall that by \cite[Proposition 3.38]{wonpic}, the homomorphisms between rank one projective modules (viewed as embedded in $\QgrA$) are all given by multiplication by an element in the commutative ring $\kk(z)$. Since the autoequivalences $\iota_{L}$ act on morphisms by restriction, we need only check that multiplication by 
$h_{I \cap J}^{-1} h_{(I \oplus J) \cap K}^{-1}$ is the same as $h_{J \cap K}^{-1} h_{I \cap (J\oplus K)}^{-1}$. This is true since
\[ (I \cap J) \cap (I \oplus J) \cap K = (J \cap K) \cap I \cap (J \oplus K) = \emptyset
\]
and
\[(I \cap J) \cup \left( (I \oplus J) \cap K \right) = (I \cap J) \cup (I \cap K) \cup (J \cap K) = (J \cap K) \cup ( I \cap (J \oplus K)) .
\]
Therefore, we conclude that the isomorphisms $\{\Theta_{I,J} \mid I,J \in \ZZ_{\fin} \}$ and autoequivalences $\{\iota_K \mid K \in \ZZ_{\fin}\}$ satisfy condition~\eqref{thetacond}.
\end{proof}

As in Proposition~\ref{autoring}, we can define the $\ZZ_{\fin}$-graded ring
\[ B = \bigoplus_{J \in \ZZ_{\fin}} B_J = \bigoplus_{J \in \ZZ_{\fin}} \Hom_{\gr A}(A, \iota_JA).
\]
To be explicit, the multiplication in $B$ is defined as follows. For $a \in B_{I}$ and $b \in B_{J}$, $a \cdot b \in \Hom_{\gr A}(A,\iota_{I \oplus J} A)$ is defined by
\[ a \cdot b = \iota_{I \oplus J}(\sigma_{I \cap J}) \circ \iota_J(a) \circ b.
\]

\begin{theorem}\label{Bequivalent} Let $\m = 0$ so $f = z^2$. There is an equivalence of categories 
\[\gr A(f) \equiv \grB.\]
\end{theorem}
\begin{proof} This is an immediate corollary of Propositions~\ref{autoring} and \ref{RSequiv} together with Lemma~\ref{Btheta}.
\end{proof}

Our next results describe some properties of the ring $B$ which will allow us to give a presentation for $B$. We first establish some notation. For $J \in \ZZ_\fin$ let $\varphi_J$ be the map
\[ \varphi_J: (\iota_JA)_0 \sra \Hom_{\gr A}(A,\iota_J A)
\]
which takes $m \in (\iota_J A)_0$ to the homomorphism defined by $\varphi_J(m)(a) = m\cdot a$. It is clear that $\varphi_J$ is an isomorphism of $\kk[z]$-modules, and we will use $\varphi_J$ to identify $B_J$ with $(\iota_J A)_0$. For $J \in \ZZ_\fin$, define $b_J := \varphi_J(h_J) \in B_J$ with $b_\emptyset := \varphi_\emptyset(1)$. 

\begin{lemma}\label{mBprops}Let $\m = 0$ and let $I, J \in \ZZ_\fin$.
 \begin{enumerate} \item \label{mBprops1} $(\iota_J A)_0 = h_J \kk[z]$ so $b_J$ freely generates $B_J$ as a right $B_{\emptyset}= \kk[z]$-module.
\item \label{mBprops3} $b_I b_J = b_{I \cap J}^2 b_{I \oplus J}$ so $b_J = \prod_{j \in J} b_j$.
\item For all $n \in \ZZ$, $b_n^2 = \varphi_\emptyset(h_n)$.
\item $B$ is a commutative $\kk$-algebra generated by $\{b_n \mid n \in \ZZ\}$.
\end{enumerate}
\end{lemma}
\begin{proof}
This is an analogue of \cite[Lemma 10.2]{smith}; we use similar arguments to Smith, altering them slightly when necessary. Though these results are similar, we will see later in this section that, interestingly, the ring $B$ exhibits properties that are rather different those of from Smith's ring $C$.

 (1) Recall that in \cite[Proposition 5.13]{wonpic}, for each $n \in \ZZ$ and for each rank one projective $P$, we constructed $\iota_n P$ as a submodule of $P$, in particular the kernel of a surjection $P \sra A/zA \s{n}$ or $P \sra A/yA\s{n+1}$. We also saw that $(\iota_n^2 P)_0 = (z+n)^2 P_0$. Therefore, $(z+n)^2 A_0 = (\iota_n^2 A)_0 \subseteq (\iota_n A)_0$. But since $\iota_n A$ is the kernel of the nonzero morphism to $A/zA\s{n}$ or $A/yA\s{n+1}$, and these modules have $\kk$-dimension $2$ in all graded components where they are nonzero \cite[Lemma 3.18]{wonpic}, so $(\iota_n A)_0$ has $\kk$-codimension 2 in $A_0$. Hence, 
\[(\iota_n A)_0 = (z+n)^2 A_0 = h_n \kk[z].
\]
Now, since the autoequivalences $\iota_n$ commute, $(\iota_J A)_0 \subseteq (\iota_j A)_0 = h_j \kk[z]$ for each $j \in J$. Additionally, $(\iota_J A)_0$ has $\kk$-codimension $2|J|$ in $A_0$. Therefore, 
\[(\iota_J A)_0 = h_J \kk[z].\]

Identifying $(\iota_J A)_0$ with $B_J$ via $\varphi$, we see that $b_J = \varphi_J(h_J)$ freely generates $B_J$ as a $\kk[z]$-module. Now since $\iota_\emptyset = \Id_{\gr A}$, multiplication $B_J \times B_\emptyset \sra B_J$ sends $(f,g)$ to $f \circ g$. Since $B_\emptyset = \kk[z]$, the result follows.

(2) This result follows from a proof identical to the proof of \cite[Lemma 10.2.(5)-(6)]{smith}. For convenience, we summarize it here. By the definition of multiplication in $B$,
\[b_I b_J = \iota_{I \oplus J}(\sigma_{I \cap J}) \circ \iota_J(\varphi_I(h_I)) \circ \varphi_J(h_J).
\]
Recalling that the involutions $\{\iota_I \mid I \in \ZZ_\fin\}$ act on morphisms by restriction, we see that $b_I b_J: A \sra \iota_{I \oplus J} A$ is given by left multiplication by $h_{I \cup J}$, and therefore
\begin{align*} b_I b_J = \varphi_{I \oplus J}(h_{I \cup J}).
\end{align*}
Now note that 
\begin{align*} b_{I \cap J}^2b_{I \oplus J} &= b _{I \cap J}(b_{I \cap J} b_{I \oplus J}) = b_{I\cap J}(\varphi_{I \cup J}(h_{I \cup J})) = b_{I \cap J} b_{I \cup J} \\
&= \varphi_{(I \cap J) \oplus (I \cup J)}(h_{(I \cap J) \cup (I \cup J)}) \\
&= \varphi_{I \oplus J}(h_{I \cup J}) = b_ Ib_J.
\end{align*}

Induction on $|J|$ yields $b_J = \prod_{j \in J} b_j$.

(3) For $a \in A$,
\[ (b_n.b_n)(a) = \sigma_n(h_n^2) (a) = (z+n)^{-2}h_n^2a = h_n a. 
\]
Hence, $b_n^2$ is given by multiplication by $(z+n)^2$, that is, 
\[b_n^2 = \varphi_\emptyset((z+n)^2) . \]

(4) Notice that 
\[b_1^2 - b_0^2 = \varphi_\emptyset((z+1)^2 - z^2) = \varphi_\emptyset(2z + 1)
\] 
so
\[\varphi_\emptyset(z) =\frac{1}{2}\left( b_1^2 - b_0^2 - \varphi_\emptyset(1)\right) .
\]
Hence, the $b_n$ generate $B_\emptyset$ as a $\kk$-algebra, and combined with parts \ref{mBprops1} and \ref{mBprops3}, the $b_n$ generate $B$ as a $\kk$-algebra. By part \ref{mBprops3}, $b_n b_m = b_m b_n$ for all $n, m \in \ZZ$, and the result follows.
\end{proof}

\begin{proposition} \label{Bpresentation} The $\ZZ_\fin$-graded ring $B$ has presentation
\[ B \cong \frac{\kk[z][b_n \mid n\in \ZZ]}{\left( b_n^2 = (z+n)^2 \mid n \in \ZZ \right)}
\]
where $\deg z = \emptyset$ and $\deg b_n = n$.
\end{proposition}
\begin{proof}By Lemma~\ref{mBprops}, the elements $\{b_n \mid n \in \ZZ\}$ generate $B$ as a $\kk$-algebra and satisfy the relations $b_n^2 = (z+n)^2$ for all $n \in \ZZ$. Hence, we need only show that the ideal generated by these relations contains all relations in $B$.

Let $r = 0$ be a relation in $B$. Since $B$ is graded, we may assume that $r$ is homogeneous of degree $I$. By Lemma~\ref{mBprops}, we can write 
\[ r = b_I \beta = 0
\]
where $\beta$ is a $\kk[z]$-linear combination of products of $b_j^2$'s for some integers $j$. By using the relations $b_j^2 = (z+j)^2$ for each $j$, we can rewrite $\beta$ in $B$ as a polynomial $g(z) \in \kk[z]$. Hence
\[ r = b_I g(z) = 0
\]
but since $B_I$ is freely generated as a right $B_\emptyset = \kk[z]$-module by $b_I$, this implies that $g(z) = 0$, so the relation $r$ was already in the ideal $\left( b_n^2 = (z+n)^2 \mid n \in \ZZ \right)$, completing our proof.
\end{proof}

We use this presentation to prove some basic results about $B$. While the construction of $B$ was analogous to that of Smith's ring $C$ in \cite{smith}, the two rings are different enough to warrant closer examination. Smith proves that $C$ is an ascending union of Dedekind domains. In contrast, since $b_n^2 - (z+n)^2 = (b_n + z+n)(b_n - z -n)$, $B$ is not even a domain.

\begin{lemma} \label{mBminprime} An ideal of $B$ is a minimal prime ideal if and only if it is of the form
\[ \left(b_n + (-1)^{\epsilon_n}(z+n) \st n \in \ZZ \right)
\]
for some choice of $\epsilon_n \in \{0,1\}$ for each $n \in \ZZ$.
\end{lemma}
\begin{proof} We work in the polynomial ring $S = \kk[z][b_n \mid n \in \ZZ]$. The prime ideals of $B$ correspond to prime ideals of $S$ containing $\left( b_n^2 = (z+n)^2 \mid n \in \ZZ \right)$. For each $n \in \ZZ$, choose $\epsilon_n \in \{0,1\}$. Viewing $S$ as a polynomial ring with coefficients in $\kk[z]$, we see that
\[ \pp = \left(b_n + (-1)^{\epsilon_n}(z+n) \st n \in \ZZ \right)
\]
is the kernel of the map evaluating a polynomial $g(b_n \mid n \in \ZZ)$ at the point $((-1)^{\epsilon_n} (z+n) \mid n \in \ZZ)$ and so $\pp$ is a prime ideal of $S$. To see that $\pp$ corresponds to a minimal prime, we observe that for every $n \in \ZZ$, a prime ideal containing $b_n^2 - (z+n)^2$ must contain either $b_n + (z+n)$ or $b_n-(z+n)$. Hence $\pp$ corresponds to a minimal prime.
\end{proof}

For $J \in \ZZ_\fin$ we write $R_J$ for the $\kk$-subalgebra of $B$ generated by the elements $\{ 1, z \} \cup \{b_n \mid n \in J \}$. By the same argument as in Proposition~\ref{Bpresentation}, $R_J$ has the presentation
\[ R_J \cong \frac{\kk[z][b_n \mid n\in J]}{\left( b_n^2 = (z+n)^2 \st n \in J \right)}.
\]
We will use the fact that $B$ is an ascending union of the subrings $R_J$ for any ascending, exhaustive chain of subsets $J \in \ZZ_\fin$.

\begin{proposition}\label{ringBprops} $B$ is a non-noetherian, reduced ring of Krull dimension $1$.
\end{proposition}
\begin{proof}We showed in Lemma~\ref{mBminprime} that $B$ has infinitely many minimal prime ideals, so $B$ is not noetherian. Further, the quotient of $B$ by a minimal prime is isomorphic to $\kk[z]$, and hence $B$ has Krull dimension $1$. We will show that the intersection of all minimal primes is $(0)$, so the nilradical $\nn(B)=(0)$. Let $a$ be an element in the intersection of all minimal primes. We can write $a$ as a sum of finitely many homogeneous terms, so $a$ is an element of the subring $R_J$ for some $J \in \ZZ_\fin$. Suppose $j \in J$. Since 
\begin{align*} a &\in \left(b_j + (z+j)\right) + \left( b_n - (z+n) \st n \in J \setminus \{j\} \right) \mbox{ and} \\
a &\in \left(b_j - (z+j)\right) + \left( b_n - (z+n) \st n \in J \setminus \{j\} \right),
\end{align*} 
we can write
\begin{equation}\label{eqnil} a = (b_j + (z+j))r + s = (b_j - (z+j))r' + s',
\end{equation}
for some $r, r' \in R_J$ and some $s, s' \in \left( b_n - (z+n) \mid n \in J \setminus \{j\}\right)$. Setting $b_n = (z+n)$ for all $n \in J$, the right hand side of \eqref{eqnil} is identically $0$ and hence
\[ r \in \left( b_n - (z+n) \mid n\in J\right)
\]
so
\[ a = (b_j + (z+j))r + s \in \left(b_j^2-(z+j)^2\right) + \left(b_n - (z+n) \mid n \in J\setminus \{j\} \right) .
\]
Since $(b_j^2 - (z+j)^2) = (0)$ in $B$, therefore
\[ a \in \left(b_n - (z+n) \mid n \in J\setminus \{j\} \right) .
\]
Inducting on the size of $J$, we conclude that $a = 0$ and since $a$ was an arbitrary element in the intersection of all primes, we conclude that $\nn(B) = (0)$.
\end{proof}

\subsection{Non-congruent roots}
Let $\m \in \kk \setminus \ZZ$, so we are in the distinct, non-congruent root case. For notational convenience, let $\Gamma = \ZZ_\fin \times \ZZ_\fin$. This case bears resemblance to the multiple root case. By \cite[Lemma 5.17]{wonpic}, we know that the set $\{\iota_\gamma A \mid \gamma \in \Gamma \}$ generates $\gr A$. We will show that the isomorphisms between these autoequivalences satisfy condition~\eqref{thetacond}, so we can use Proposition~\ref{autoring} to define a ring $D$ with an equivalent graded module category. We will then show that $D$ is commutative, and give a presentation for $D$.

For each $(J,J') \in \Gamma$, define the polynomial
\begin{equation} h_{(J,J')} = \prod_{ j \in J } (z+j) \prod_{j' \in J'} (z+j'+\m ).
\end{equation}
For a projective module $P$ we showed, in \cite[Proposition 5.14]{wonpic}, that $\iota_{(n ,\emptyset)}^2 P = (z+n) P$ and $\iota_{(\emptyset, n)}^2P = (z + n + \m )P$ and so $\iota_{(n, \emptyset)}^2 \cong \Id_{\gr A} \cong \iota_{(\emptyset, n)}^2$. We denote by $\sigma_{(n, \emptyset)}$ the isomorphism $\iota_{(n ,\emptyset)}^2 A \sra A$ and by $\sigma_{(\emptyset, n)}$ the isomorphism $\iota_{(\emptyset, n)}^2 A \sra A$. Since $A$ is projective, $\sigma_{(n, \emptyset)}$ is given by left multiplication by $h_{(n,\emptyset)}^{-1}$ and $\sigma_{(\emptyset, n)}$ is given by left multiplication by $h_{(\emptyset, n)}^{-1}$. Similarly, for $(J,J') \in \Gamma$, we define
\[ \sigma_{(J,J')} = \prod_{j \in J} \sigma_{(j, \emptyset)} \prod_{j' \in J'} \sigma_{(\emptyset, j')} : \iota_{(J,J')}^2 A \sra A
\]
and note that $\sigma_{(J,J')}$ is also given by left multiplication by $h_{(J,J')}^{-1}$.

Now, for $(I,I'), (J,J') \in \Gamma$, we define
\[\Theta_{(I,I'),(J,J')} = \iota_{(I \oplus J, I' \oplus J')}(\sigma_{(I \cap J, I' \cap J')}) = \Theta_{(J, J'),(I,I')}
\]
\[ \Theta_{(I,I'),(J,J')}: \iota_{(I,I')} \iota_{(J,J')} A = \iota_{(I\oplus J, I' \oplus J')} \iota_{(I \cap J, I' \cap J')}^2 A \sra \iota_{(I\oplus J, I' \oplus J')} A.
\]

\begin{lemma} \label{dBtheta} The isomorphisms $\{\Theta_{\gamma, \delta} \mid \gamma, \delta \in \Gamma \}$ and the autoequivalences $\{\iota_\gamma \mid \gamma \in \Gamma\}$ satisfy condition~\eqref{thetacond}.
\end{lemma}
\begin{proof}This result follows from the same proof as Lemma~\ref{Btheta} with doubled indices.
\end{proof}

Having checked that our autoequivalences satisfy condition~\eqref{thetacond}, we use Proposition~\ref{autoring} to define the $\Gamma$-graded ring
\[ D = \bigoplus_{\gamma \in \Gamma} D_\gamma = \bigoplus_{\gamma \in \Gamma} \Hom_{\gr A}(A, \iota_{\gamma} A) = \bigoplus_{(J,J') \in \ZZ_\fin \times \ZZ_\fin} \Hom_{\gr A}(A, \iota_{(J,J')}A).
\]

\begin{theorem}\label{dDequivalent} Let $\m = \kk \setminus \ZZ$ and let $f = z(z+ \m )$. There is an equivalence of categories 
\[\gr A(f) \equiv \gr (D, \Gamma).\]
\end{theorem}
\begin{proof} This is an immediate corollary of Propositions~\ref{autoring} and \ref{RSequiv} together with Lemma~\ref{dBtheta}.
\end{proof}

Finally, we describe properties of and give a presentation for the ring $D$. For $\gamma \in \Gamma$ let $\varphi_{\gamma}$ be the map
\[ \varphi_{\gamma}: \left(\iota_{\gamma}A\right)_0 \sra \Hom_{\gr A}\left(A,\iota_{\gamma} A\right), \quad \varphi_{\gamma}(m)(a) = m\cdot a.
\]
We use the $\kk[z]$-module isomorphism $\varphi_{\gamma}$ to identify $D_{\gamma}$ with $\left(\iota_{\gamma} A\right)_0$. For $\gamma = (J, J') \in \Gamma$ we also define
\begin{align*} c_J &= \varphi_{(J, \emptyset)}(h_{(J, \emptyset)}) \in D_{(J,\emptyset)} \mbox{ and} \\
d_{J'} &= \varphi_{(\emptyset, J')}(h_{(\emptyset, J')}) \in D_{(\emptyset,J')}.
\end{align*}
For completeness, define $c_\emptyset = d_\emptyset = 1$.

\begin{lemma}\label{dBprops}Let $\m \in \kk \setminus \ZZ$ and let $\gamma = (I, I') \in \Gamma$.
 \begin{enumerate} \item \label{dBprops1} $(\iota_{\gamma}A)_0 = h_{\gamma} \kk[z]$.
\item \label{dBprops2} $c_Id_{I'} = d_{I'}c_I$ .
\item \label{dBprops3} The element $c_{I} d_{I'}$ freely generates $D_{\gamma}$ as a $D_{(\emptyset, \emptyset)}$-module.
\item \label{dBprops4} $c_I c_J = c_{I \cap J}^2 c_{I \oplus J}$ and $d_I d_J = d_{I \cap J}^2 d_{I \oplus J}$ and so $c_I d_{I'} = \prod_{i \in I} c_i \prod_{i' \in I'} d_{i'}$.
\item \label{dBprops6} For all $n,m \in \ZZ$ 
\[c_n^2 - n = c_m^2 - m\]
\[d_n^2 - n = d_m^2 - m\]
\[c_n^2 = d_n^2 - \m . \]
\item \label{dBprops7} $D$ is a commutative $\kk$-algebra generated by $\{c_n, d_n \mid n \in \ZZ\}$.
\end{enumerate}
\end{lemma}
\begin{proof} 
The arguments in this proof are similar to those found in Lemma~\ref{mBprops} and therefore \cite[Lemma 10.2]{smith}.
(1) For each $n \in \ZZ$, in \cite[Proposition 5.14]{wonpic}, we constructed $\iota_{(n ,\emptyset)} A$ by taking the kernel of the nonzero morphism to $X_0 \s{n}$ or $Y_0 \s{n}$. Since 
\[(z+n) A_0 = (\iota_{(n ,\emptyset)}^2 A)_0 \subseteq (\iota_{(n ,\emptyset)} A)_0\]
and $(\iota_{(n ,\emptyset)} A)_0$ has $\kk$-codimension $1$ in $A_0$, therefore 
\[(\iota_{(n ,\emptyset)} A)_0 = (z+n) A_0 = (z+n) \kk[z].
\]
By an analogous argument, $(\iota_{(\emptyset, n)} A)_0 = (z+n + \m )\kk[z]$. And since the autoequivalences commute 
\[ (\iota_\gamma A)_0 \subseteq (\iota_{(i,\emptyset)} A)_0 = (z+i) \kk[z] \mbox{ and } (\iota_\gamma A)_0 \subseteq (\iota_{(\emptyset, i')} A)_0 = (z+i +\m ) \kk[z] \]
 for each $i \in I$ and $i' \in I'$. Since $(\iota_\gamma A)_0$ has $\kk$-codimension $|I| + |I'|$ in $A_0$. Therefore, 
\[(\iota_\gamma A)_0 = h_\gamma \kk[z].\]

(2) This follows since the $\iota_{\gamma}$ act on morphisms by restriction. Both $c_Jd_{J'}$ and $d_{J'}c_J$ are given by multiplication by $h_{(J, \emptyset)}h_{(\emptyset, J')} = h_{(\emptyset,J')}h_{(J,\emptyset)}$. 

(3) Since $\iota_{(\emptyset, \emptyset)} =\Id_{\gr A}$, multiplication $D_{\gamma} \times D_{(\emptyset, \emptyset)} \sra D_{\gamma}$ sends $(g,h)$ to $g \circ h$. By part \ref{dBprops1}, $D_{\gamma} = \varphi\left(h_{\gamma} \kk[z]\right)$. Since $D_{(\emptyset, \emptyset)} = \kk[z]$ it follows that $D_{\gamma}$ is generated as a right $D_{(\emptyset, \emptyset)}$-module by $\varphi_{\gamma}\left(h_{\gamma}\right)$, or left multiplication by $h_{\gamma}$. Now by the definition of multiplication in $D$,
\[ c_Id_{I'} = \iota_{(\emptyset, I')}\left(\varphi(h_{(I, \emptyset)})\right) \varphi\left(h_{(\emptyset, I')} \right).
\]
and since $\iota_{(I, \emptyset)}$ acts on morphisms by restriction, $c_Id_{I'}$ is multiplication by $h_{(I,\emptyset)} h_{(\emptyset, I')} = h_{\gamma}$.

(4) This has the same proof as part~\ref{mBprops3} of Lemma~\ref{mBprops}.

(5) For $a \in A$,
\[ (c_n.c_n)(a) = \sigma_{(n,\emptyset)}\left((z+n)^2\right) (a) = (z+n)^{-1}(z+n)^2a = (z+n)a. 
\]
Hence, $c_n^2$ is given by multiplication by $z+n$, that is, 
\[c_n^2 = \varphi_{(\emptyset, \emptyset)}(z+n) . \]
Similarly,
\[d_n^2 = \varphi_{(\emptyset, \emptyset)}(z+n+ \m ),
\]
from which the claim follows.

(6) This follows from parts \ref{dBprops2}, \ref{dBprops3}, and \ref{dBprops4}. 
\end{proof}

\begin{proposition} Let $\m \in \kk \setminus \ZZ$. The $\Gamma$-graded ring $D$ has presentation
\[ D \cong \frac{\kk[c_n, d_n \mid n\in \ZZ]}{\left( c_n^2 - n = c_m^2 - m, c_n^2 = d_n^2 - \m \mid m,n \in \ZZ\right)}
\]
where $\deg c_n = (n, \emptyset)$ and $\deg d_n = (\emptyset, n)$.
\end{proposition}
\begin{proof}By Lemma~\ref{dBprops}, the elements $\{c_n\}$ and $\{d_n\}$ generate $D$ as a $\kk$-algebra and satisfy the relations $c_n^2 -n = c_m^2 -m$ and $c_n^2 = d_n^2 -\m $ for all $n,m \in \ZZ$. We need to show that the ideal generated by these relations contains all relations in $D$.

Let $r = 0$ be a relation in $D$. We may assume that $r$ is homogeneous of degree $(I,I')$. By Lemma~\ref{dBprops}, we can write 
\[ r = c_I d_{I'}\beta = 0
\]
where $\beta$ is a $\kk[z]$-linear combination of products of $c_j^2$'s and $d_{j}^2$'s for some integers $j$. By using the relations $c_j^2 = (z+j)$ and $d_{j}^2 = c_j^2 + \m $ for each $j$, we can rewrite $\beta$ in $D$ as a polynomial $g(z) \in \kk[z]$. Hence
\[ r = c_I d_{I'} g(z) = 0
\]
but since $B_{(I,I')}$ is freely generated as a right $B_{(\emptyset, \emptyset)} = \kk[z]$-module by $c_I d_{I'}$, this implies that $g(z) = 0$, so the relation $r$ was already in the ideal 
\[\left( c_n^2 - n = c_m^2 - m, c_n^2 = d_n^2 - \m \mid m,n \in \ZZ\right),\]
 completing our proof.
\end{proof}

\begin{remark}\label{halfcase} In the case that $\m = 1/2$, the ``picture" of the simple modules of $\gr A(f)$ looks like the picture of $\gr A_1$ given in Figure~\ref{fig.grA}, scaled by a factor of $1/2$. In \cite[Proposition 5.9]{wonpic}, the author was able to construct an autoequivalence of $\gr A(f)$ $\SSS^{1/2}$ whose square was naturally isomorphic to the shift functor $\SSS_{A(f)}$. This autoequivalence corresponds to shifting the picture of the simple modules by $1/2$. We note that in this case, the above ring $D$ is isomorphic to Smith's ring in \cite{smith}
\[ R = \frac{\kk[x_n \mid n \in \ZZ]}{\left(x_n^2 -n = x_m^2 +m \mid n, m \in \ZZ\right)}
\]
via the non-unital ring isomorphism $\varphi: D \to R$ where $\varphi(1)  = 2$, $\varphi(c_n) = x_{-2n}$, and $\varphi(d_n) = x_{-(2n+1)}$. This is a graded isomorphism if we identify $\ZZ_{\fin} \times \ZZ_{\fin}$ with $\ZZ_{\fin}$ where $(\{a_1, \dots a_n\}, \{b_1, \dots b_m\}$ is identified with $\{-2a_1, \dots -2a_n \}, \{-(2b_1 + 1), \dots -(2b_m + 1)\}$.
\end{remark}

\subsection{Congruent roots and the quotient category qgr-\emph{A}}

Let $\m \in \NNp$ so that we are in the congruent root case. The fact that the set $\{ \iota_J A \mid J \in \ZZ_\fin\}$ does not generate $\gr A$ is a significant difference from the other cases. In particular, we are unable to use Proposition~\ref{RSequiv}. One of the main obstructions is that there are infinitely many orbits of rank one graded projectives under the action of numerically trivial autoequivalences since numerically trivial autoequivalences fix the finite-dimensional simple modules, $Z\s{i}$. 

Recall that the \textit{quotient category} $\CCC/\DDD$ of an abelian category by a Serre subcategory was first defined by Gabriel in his thesis \cite{gabriel} as follows. The objects of $\CCC/\DDD$ are the objects of $\CCC$. If $M$ and $N$ are objects of $\CCC$, then
\[ \Hom_{\CCC/\DDD}(M,N) := \dirlim \Hom_{\CCC}(M', N/N'),
\]
where the direct limit is taken over all subobjects $M' \subseteq M$ and $N' \subseteq N$ where $M/M' \in \DDD$ and $N' \in \DDD$. Composition of morphisms in $\CCC/\DDD$ is induced by composition in $\DDD$. There is also a canonical exact functor, the \textit{quotient functor} $\pi: \CCC \sra \CCC/\DDD$. For an object $M$, $\pi M = M$ and for a morphism $f$, $\pi f$ is given by the image of $f$ in the direct limit.

We will consider the quotient category of $\gr A$ modulo its full subcategory of finite-dimensional modules. This is the same construction that Artin and Zhang \cite{AZ} use in their definition of the noncommutative projective scheme associated to an $\NN$-graded ring $R$. However, as $A$ is $\ZZ$-graded, the details are somewhat different. We will investigate this quotient category fairly explicitly.

Let $\fdim A$ denote the full subcategory of $\gr A$ consisting of all finite-dimensional modules. The only finite-dimensional simple modules are the shifts of $Z$ (see the start of section~\ref{catgrA}) so each object in $\fdim A$ has a composition series consisting entirely of shifts of $Z$. Since $A$ is a homomorphic image of $U(\mathfrak{sl}(2,\kk)$, every finite-dimensional $A$-module is semisimple \cite[8.5.12(ii)]{mcconnell}. Therefore every object in $\fdim A$ is a direct sum of shifts of $Z$. Since $\fdim A$ is a Serre subcategory, we can define
\[\qgr A = \gr A / \fdim A.\]
The next lemma allows us to write down a $\Hom$ set in $\qgr A$ in a very concrete way.

\begin{lemma} \begin{enumerate} \item \label{cdirlim} For every finitely generated graded right $A$-module $M$, there exists a unique smallest submodule $\kappa(M) \subseteq M$ such that $M/\kappa(M) \in \fdim A$.
\item For every finitely generated graded right $A$-module $N$, there exists a unique largest submodule $\tau (N) \subseteq N$ such that $\tau(N) \in \fdim A$.
\end{enumerate}
\end{lemma} 
\begin{proof} (1) Since $M$ is finitely generated, $\uHom_A(M,Z)$ is finite-dimensional over $\kk$ and is nonzero in only finitely many degrees, say $d_1,\ldots, d_n$. Let $\III = \{ Z\s{d_1} \ldots, Z\s{d_n}\}$ and $\DDD$ be the full subcategory of $\gr A$ consisting of all finite direct sums of elements from $\III$. By \cite[Proposition 4.6]{won}, there exists a unique smallest submodule $M'$ such that $M/M' \in \DDD$. In particular, by the argument in \cite[Proposition 4.6]{won}, $M'$ is the intersection of all kernels of maps $M \sra Z\s{i}$. Let $\kappa (M) = M'$. Note that in fact $\kappa (M)$ is the unique smallest submodule such that $M/\kappa(M) \in \fdim A$ because $M/\kappa(M) \in \fdim A$ and any factor of $M$ in $\fdim A$ is also in $\DDD$.

(2) Since $N$ is finitely generated and $A$ is noetherian, $N$ is noetherian. Since the sum of two objects in $\fdim A$ is again an object in $\fdim A$, there exists a unique largest submodule of $N$ that is a direct sum of shifts of $Z$. Call this largest submodule $\tau (N)$. 
\end{proof}

The preceding lemma allows us to describe $\Hom_{\qgr A}(M,N)$ without any reference to a direct limit. In particular,
\[ \Hom_{\qgr A}(\pi M,\pi N) = \Hom_{\gr A}(\kappa(M), N/\tau(N)).
\]

\begin{proposition} \label{qgrproj} Let $P$ be a projective graded right $A$-module. Then $\pi P$ is projective in $\qgr A$ if and only if $\kappa(P) = P$.
\end{proposition}
\begin{proof}Suppose $P$ is projective in $\gr A$ and $\kappa(P) = P$. Let $M$ and $N$ be graded $A$-modules and let $g \in \Hom_{\qgr A}(\pi M, \pi N)$ be an epimorphism. We show that for any morphism $h: \pi P \sra \pi N$, there exists a lift $j: \pi P \sra \pi M$ so that $\pi P$ is projective. 

To show projectivity of $\pi P$, we are only concerned with these morphisms in $\qgr A$. Therefore, we may replace $M$ by $\kappa(M)$ since $\pi M = \pi \kappa(M)$ and similarly we may replace $N$ by $N/ \tau(N)$. Then $g$ is represented by a morphism $\ol{g} \in \uHom_A(M, N)$ such that $\coker (\ol g) \in \fdim A$. Since $\coker(\ol{g}) \in \fdim A$, $\pi N = \pi \im \ol{g}$, and so we may replace $N$ with $\im \ol{g}$ and assume that $\ol{g}$ is surjective.

Now note that since $\kappa(P) = P$, we have $\Hom_{\qgr A}(\pi P, \pi N) = \uHom_A(P,N)$. So $h$ is represented by a morphism $\ol{h}: P \sra N$, which by the projectivity of $P$ in $\gr A$ lifts to a morphism $ \ol{j}: P \sra M$. Let $j = \pi(\ol{j})$.

Conversely, suppose that $\kappa(P) \neq P$. Then $\kappa(P)$ fits into the exact sequence
\[ 0 \ra \kappa(P) \ra P \ra \bigoplus_{i \in I} Z \s{i} \ra 0
\]
for some finite set of integers $I$. By \cite[Theorem 5]{bav}, $Z$ has projective dimension $2$, so this exact sequence shows that $\kappa(P)$ is not projective. Since this sequence is exact, $\pd (\kappa(P)) \leq \max \{ \pd (P), \pd (Z) - 1\}$. Therefore, $\kappa(P)$ has projective dimension 1 so has a projective resolution in $\gr A$:
\begin{equation}\label{kappaexact} 0 \ra P_1 \overset{d_1}{\ra} P_0 \overset{d_0}{\ra} \kappa(P) \sra 0.
\end{equation}

Suppose for contradiction that $\pi P = \pi\kappa(P)$ is projective in $\qgr A$. Since $\pi$ is an exact functor, we have the following exact sequence in $\qgr A$:
\[ 0 \ra \pi P_1 \overset{\pi d_1}{\ra} \pi P_0 \overset{\pi d_0}{\ra} \pi \kappa(P) \sra 0.
\]
Since we assumed $\pi \kappa(P)$ is projective, there exists a splitting $h: \pi \kappa(P) \sra \pi P_0$ such that $\pi d_0 \circ h = \Id_{\pi \kappa(P)}$. The splitting $h$ is represented by a morphism $\ol{h} \in \uHom_A( \kappa(P), P_0)$, since $P_0$ has no finite-dimensional submodules. Now $\ol{h}$ gives a splitting of \eqref{kappaexact}, since $\pi(\ol{h} \circ d_0) = \Id_{\pi \kappa(P)}$ and since $\kappa(P)$ has no finite-dimensional submodules, we know that $\Hom_{\qgr A}(\pi \kappa(P), \pi \kappa(P)) = \uHom_A(\kappa(P), \kappa(P))$. Hence, $\ol{h} \circ d_0 = \Id_{\kappa(P)}$. But a splitting of \eqref{kappaexact} shows that $\kappa(P)$ is projective, which is a contradiction.
\end{proof}

\begin{corollary} \label{qgrproj2} If $P$ is a rank one projective in $\gr A$ then $\pi P$ is projective in $\qgr A$ if and only if for each $n \in \ZZ$, $P$ surjects onto exactly one of $X\s{n}$ or $Y \s{n}$.
\end{corollary}
\begin{proof} This follows immediately from Proposition~\ref{qgrproj} and Lemma~\ref{projXYZ}.
\end{proof}

\begin{proposition}\label{qgrgenerator} Let $\PPPP = \{P_i\}_{i \in I}$ be a set of rank one projective modules in $\gr A$. If $\PPPP$ generates every shift of $X$ and $Y$ then $\pi \PPPP = \{\pi P_i\}_{i \in I}$ generates $\qgr A$. In particular, $\qgr A$ has enough projectives.
\end{proposition}
\begin{proof}Suppose $\PPPP$ is a set of projective modules in $\gr A$ which generates every shift of $X$ and $Y$. Since the shifts of $A$ generate $\gr A$, likewise the shifts of $\pi A$ generate $\qgr A$. Since $\pi A = \pi \kappa(A)$, we will show that $\PPPP$ generates every shift of $\kappa(A)$ and hence $\pi \PPPP$ generates every shift of $\pi \kappa (A)$. 

Let $P \in \PPPP$ and choose a maximal embedding $\varphi: P \sra \kappa(A) \s{n}$. It suffices to construct a surjection ${\psi: \bigoplus_{j \in J} P_j \sra \kappa(A)/P}$ for some $J\subseteq I$. This is because, by the projectivity of the $P_j$, there exists a lift $\overline{\psi}: \bigoplus_{j \in J} P_j \sra \kappa(A)\s{n}$ and because $\im \varphi + \im \overline{\psi} = \kappa(A)\s{n}$.

Since $A$ has Krull dimension 1, the quotient $\kappa(A)\s{n}/P$ has finite length. Further, we know that $\uHom_A(\kappa(A) , Z) = 0$, so $\kappa(A)\s{n}/P$ is a direct sum of indecomposables, none of which has a factor of $Z$. It thus suffices to show that $\PPPP$ generates every such indecomposable. 

Without loss of generality, suppose $\kappa(A)\s{n}/P$ is an indecomposable with a factor of $X \s{i}$. We induct on the length of the indecomposable. By hypothesis, some $P_0 \in \PPPP$ surjects onto $X\s{i}$. By the projectivity of $P_0$, this surjection then lifts to a map $g: P_0 \sra \kappa(A)\s{n}/P$. If this is a surjection, then we are done. Otherwise, $P_0$ surjects onto a proper submodule of $\kappa(A)\s{n}/P$. Again, it suffices to give a surjection onto the cokernel of $g$. But now note that since $\kappa(A)\s{n}$ surjects onto $(\kappa(A)\s{n}/P)/\im(g)$, then $(\kappa(A\s{n}/P)/\im(g)$ again has no factor of $Z$ and has shorter length. By induction, $\PPPP$ generates $(\kappa(A)\s{n}/P)/\im(g)$ and thus $\PPPP$ generates $\kappa(A)\s{n}$.

Now $\qgr A$ has enough projectives by Corollary~\ref{qgrproj2} and \cite[Lemma 3.36]{wonpic}.
\end{proof}

\begin{theorem}\label{qgrequivgr} Let $\m \in \NNp$. There is an equivalence of categories 
\[\gr A\left(z^2\right) \equiv \qgr A\left(z(z+\m )\right).\]
\end{theorem}
\begin{proof} Let $f = z^2$ and $g = z(z+\m )$. Let $\PPP$ be the full subcategory of $\gr A(f)$ consisting of direct sums of the canonical rank one projective $A(f)$-modules (as described after Lemma~\ref{projXYZ}). Let $\PPP'$ be the full subcategory of $\qgr A(g)$ consisting of the images in $\qgr A(g)$ of direct sums of the canonical rank one projectives of $\gr A(g)$ that remain projective in $\qgr A(g)$. We will define an equivalence of categories $\GGG: \PPP \sra \PPP'$. We will then use \cite[Lemma 4.2]{wonpic} to extend this to an equivalence $\gr A(f) \equiv \qgr A(g)$.

First, we define $\GGG$ on objects. Let $P$ be a canonical rank one projective $A(f)$-module. By Lemma~\ref{projXYZ}, for each $n \in \ZZ$, $P$ surjects onto exactly one of $X^f\s{n}$ and $Y^f\s{n}$. Define $P'$ to be the canonical projective object of $\gr A(g)$ with simple factors corresponding to those of $P$, that is, for all $n \in \ZZ$, if $F_n(P) = X^f\s{n}$ then $F_n(P') = X^g\s{n}$ and if $F_n(P) = Y^f\s{n}$ then $F_n(P') = Y^g\s{n}$. Such a projective $P'$ exists by \cite[Lemma 3.36]{wonpic}. By Corollary~\ref{qgrproj2}, $\pi P'$ is a projective object of $\qgr A(g)$. Now define $\GGG(P) = \pi P'$. By abuse of notation, we will also refer to the object $P'$ of $\gr A(g)$ as $\GGG(P)$. For a direct sum of canonical rank one projectives, $P = \bigoplus_{i \in I} P_i$, define $\GGG(P) := \bigoplus_{i \in I} \GGG(P_i)$.

Suppose now that $P = \bigoplus_{i \in \ZZ} (p_i)x^i$ and $Q = \bigoplus_{i \in \ZZ} (q_i)x^i$ are canonical rank one projectives of $\gr A(f)$ with structure constants $\{c_i\}$ and $\{d_i\}$, respectively. By \cite[Lemma 3.40]{wonpic}, $\Hom_{\gr A(f)}(P,Q)$ is generated as a $\kk[z]$-module by left multiplication by $\theta_{P,Q} = \lcm_{i \in \ZZ}(q_i/\gcd(p_i,q_i))$. By construction, $\GGG(P)$ and $\GGG(Q)$ have no finite-dimensional submodules. Hence, $\Hom_{\qgr A(g)}\left(\GGG(P), \GGG(Q)\right) \cong \Hom_{\gr A(g)}\left(\GGG(P), \GGG(Q)\right)$ is also generated as a $\kk[z]$-module by some maximal embedding. We will show that in fact this maximal embedding is given by multiplication by the same $\theta_{P,Q}$.

By \cite[Lemma 3.10]{wonpic} and \ref{projXYZ}, for every $i \in \ZZ$, $c_i = 1$ if and only if $F_{i}(P) = X^f\s{i}$, and $c_i = (z+i)^2$ if and only if $F_{i}(P) = Y^f\s{i}$. The same is true for the structure constants $\{d_i\}$ of $Q$. Now let $I = \{i_1, \ldots, i_r\} \subset \ZZ$ be precisely those indices $i_j$ such that $F_{i_j}(P) = X^f\s{i_j}$ and $F_{i_j}(Q) = Y^f\s{i_j}$. These are the indices $i_j$ such that $c_{i_j} = 1$ and $d_{i_j} = (z + i_j)^2$. Observe that for distinct integers $n$ and $m$ then $c_n$ and $d_n$ are relatively prime to $c_m$ and $d_m$ so that 
\[ \theta_{P,Q} = \lcm_{i \in \ZZ} \frac{q_i}{\gcd\left(p_i,q_i\right)} = \prod_{i_j \in I} (z+i_j)^2 .
\]

Let $\{c_i'\}$ and $\{d_i'\}$ be the structure constants for $P'= \GGG(P) = \bigoplus_{i\in\ZZ} (p_i') x^i $ and $Q' = \GGG(Q) = \bigoplus_{i \in \ZZ}(q_i') x^i$. As in \cite[Table 1]{wonpic}, for each $n \in \ZZ$, we can calculate $c_n'$ from the simple factors of $P'$. Specifically $c_n'$ depends only on $F_n(P')$ and $F_{n + \m }(P')$. For any integer $n$, the polynomial $z+n$ is only possibly a factor of $c_n$ or $c_{n-\m }$ and these structure constants depend only on the simple factors at $n$, $n+ \m $ and $n- \m $. We use \cite[Table 1]{wonpic} and the fact that no shift of $Z$ is a factor of $P'$ to construct the following table:

\begin{table}[h!]
\centering
\caption{The structure constants of $P'$ in terms of its simple factors.}
\begin{tabular}{c|cc}
 & $F_n(P') = X^g\s{n}$ & $F_n(P') = Y^g\s{n}$ \\ \hline
$F_{n+\m }(P') = X^g\s{n + \m }$ & $c_n' = 1$ & $c_n' = \sigma^{n}(z)$ \\
$F_{n+\m }(P') = Y^g\s{n + \m }$ & $c_n' = \sigma^{n }(z + \m )$ & $c_n' = \sigma^{n }(f)$ \\
$F_{n-\m }(P') = X^g\s{n- \m }$ & $c_{n - \m }' = 1 $ & $c_{n - \m }' = \sigma^{n-\m }(z+ \m )$ \\
$F_{n-\m }(P') = Y^g\s{n- \m }$ & $c_{n- \m }' = \sigma^{n- \m }(z) $ & $c_{n - \m }' = \sigma^{n-\m }(f) $
\end{tabular}
\end{table}

Observe from the table that either
\begin{enumerate}[(i)]
\item $F_n(P') = Y^g\s{n}$ in which case $z+n$ is a factor of both $c_n'$ and $c_{n- \m }'$, or else
\item $F_n(P') = X^g\s{n}$ in which case $z+n$ is not a factor of either $c_n'$ or $c_{n-\m }'$.
\end{enumerate}

Now when we calculate the maximal embedding $P' \sra Q'$
\[ \theta_{P',Q'} = \lcm_{i \in \ZZ} \frac{q_i'}{\gcd\left(p_i',q_i'\right)}
\]
we see that for $i_j \in I$, $q_{i_j - \m }$ has a factor of $(z+i_j)^2$ while $p_{i_j - \m}$ has no factor of $(z+i_j)$. For all other integers $n$, the factor $(z+n)$ appears either only as a factor of $p_n$ and $p_{n-\m }$ or else as a factor of $p_n$, $p_{n - \m }$, $q_n$, and $q_{n- \m }$. Hence, 
\[ \theta_{P',Q'} = \prod_{i_j \in I} (z+i_j)^2 = \theta_{P,Q}.
\]
So the maximal embeddings $P\sra Q$ and $\GGG(P) \sra \GGG(Q)$ are both given by multiplication by the same element $\theta_{P,Q}$. If $\varphi \in \Hom_{\gr A(f)}(P,Q)$ is given by multiplication by $\theta_{P,Q} \beta$ then define $\GGG(\varphi) \in \Hom_{\qgr A(g)}(\GGG(P), \GGG(Q))$ to be multiplication by $\theta_{P,Q} \beta$ also. By our definition of $\GGG$ on morphisms, it is clear that $\GGG(\Id_{\gr A(f)})$ is the identity and $\GGG$ respects composition. 

Altogether, we have defined $\GGG$ on the canonical rank one projectives of $\gr A(f)$. By \cite[Lemma 4.3]{wonpic}, we can extend $\GGG$ to a functor $\GGG: \PPP \sra \PPP'$. It is easily seen that $\GGG$ is an equivalence on these subcategories. To see that $\GGG$ is full and faithful, notice that for canonical rank one projectives $P$ and $Q$ of $\gr A(f)$ our construction of $\GGG$ gave an isomorphism 
\[\Hom_{\gr A(f)}(P,Q) \cong \Hom_{\qgr A(g)}(\GGG(P), \GGG(Q)).\]
Now given direct sums of canonical rank one projectives $\bigoplus_{i\in I} P_i$ and $\bigoplus_{j\in J} Q_j$, the construction in \cite[Lemma 4.3]{wonpic} gives an isomorphism
\begin{align*}&\Hom_{\gr A(f)}\left(\bigoplus_{i\in I} P_i,\bigoplus_{j\in J} Q_j\right) \cong \bigoplus_{i\in I} \bigoplus_{j\in J} \Hom_{\gr A(f)}\left( P_i, Q_j\right) \\
&\cong  \bigoplus_{i\in I} \bigoplus_{j\in J} \Hom_{\qgr A(g)}\left(\GGG(P_i), \GGG(Q_j)\right) \cong\Hom_{\qgr A(g)}\left(\GGG\bigoplus_{i\in I} P_i, \GGG\bigoplus_{j\in J} Q_j\right).
\end{align*}
To see that $\GGG$ is essentially surjective, notice that given $P'$ in $\PPP'$, we can construct a module $P$ of $\gr A(f)$ such that $\GGG(P) \cong P'$ by constructing a direct sum of canonical rank one projectives in $\gr A(f)$ with corresponding simple factors. 

By \cite[Proposition 3.43]{wonpic}, every object of $\gr A(f)$ has a projective resolution by objects of $\PPP$. Similarly, by Proposition~\ref{qgrgenerator}, every object of $\qgr A(g)$ has a projective resolution by objects of $\PPP'$. Hence, by \cite[Lemma 4.2]{wonpic} there is an equivalence $\gr A(f) \equiv \qgr A(g)$.
\end{proof}

\begin{corollary} \label{cBequiv} Let $\m \in \NNp$. Then there is an equivalence of categories 
\[\qgr A\left( z (z+\m) \right) \equiv \operatorname{gr}(B,\ZZ_\fin).\]
\end{corollary}
\begin{proof} This follows immediately from Theorem~\ref{qgrequivgr} and Theorem~\ref{Bequivalent}.
\end{proof}

\bibliography{refs}{}
\bibliographystyle{amsalpha}

\end{document}